\documentclass[12pt]{amsart}
\usepackage{mathrsfs}

\usepackage[T1]{fontenc} % Better font encoding (supports accents, hyphenation)
\usepackage{lmodern} % Latin Modern fonts (vector version of Computer Modern)
\usepackage{microtype} % Subtle typographic enhancements (kerning, spacing)

\usepackage{mathtools} % Enhancements to amsmath (recommended over amsmath)
\usepackage{amssymb} % Extra math symbols
\usepackage{mathbbol} % Extended \mathbb{} support (digits, lowercase)

\usepackage[dvipsnames]{xcolor} % Named colors (e.g., NaviBlue, ForestGreen)

\usepackage{tikz-cd} % Commutative diagrams

\usepackage{csquotes} % Context-aware quotation marks via \enquote{}

\usepackage{xargs} % Multiple optional args with \newcommandx

\usepackage{enumerate} % Custom labels in numbered lists
\usepackage{enumitem} % Full control over list spacing, labels, etc.

\usepackage{todo} % Package for notes, to-do boxes, and remarks

\usepackage{tikz} % Drawing package
\usepackage{tikz-3dplot} % 3D plots in TikZ
\usetikzlibrary{shapes, positioning, calc, arrows}

\usepackage{todo}

\usepackage[
bookmarks=true,
linktocpage=true,
bookmarksnumbered=true,
breaklinks=true,
pdfstartview=FitH,
hyperfigures=false,
plainpages=false,
naturalnames=true,
colorlinks=true,
pagebackref=false,
pdfpagelabels
]{hyperref}

\hypersetup{
	colorlinks,
	citecolor=NavyBlue,
	filecolor=NavyBlue,
	linkcolor=NavyBlue,
	urlcolor=NavyBlue
}

\usepackage{xr-hyper}
\usepackage[capitalise,noabbrev,nosort]{cleveref}
\newcommand{\sSet}{\mathsf{sSet}}
\newcommand{\Set}{\mathsf{Set}}

\renewcommand{\th}{\mathrm{th}}
\newcommand{\op}{\mathrm{op}}
\newcommand{\simplex}{\triangle}
\newcommand{\gsimplex}{\mathbb{\Delta}}
\newcommand{\sing}{\mathrm{Sing}}
\newcommand{\defeq}{\coloneqq}

\theoremstyle{plain}
\newtheorem{theorem}{Theorem}
\newtheorem{corollary}[theorem]{Corollary}
\newtheorem{lemma}[theorem]{Lemma}

\theoremstyle{definition}

\theoremstyle{remark}

\newcommand{\id}{\textup{id}}

\newcommand{\Hom}{\textup{Hom}}

\newcommand{\xr}{\xrightarrow}

\newcommand{\Z}{\mathbb{Z}}

\newcommand{\R}{\mathbb{R}}

\newcommand{\mc}[1]{\mathcal{#1}}

\newcommand{\defn}[1]{{\color{Sepia} \emph{#1}}}

\DeclarePairedDelimiter\set{\{}{\}}

\newcommand{\bd}{\partial}

\newcommand{\diag}{\mathrm{d}}

\title{The transverse singular complex}

\author[G. Friedman]{Greg Friedman}
\address{Texas Christian University}
\email{g.friedman@tcu.edu}

\author[A. Medina-Mardones]{Anibal M. Medina-Mardones}
\address{Western University}
\email{anibal.medina.mardones@uwo.ca}

\author[D. Sinha]{Dev Sinha}
\address{University of Oregon}
\email{dps@uoregon.edu}

\date{\today}



\begin{document}
%	\begin{center}
%		---PREPRINT---
%		\vskip 20pt
%	\end{center}
	\begin{abstract}
	Let $M$ be a smooth manifold without boundary and let $\mc T$ be a countable collection of manifolds with corners, each equipped with a smooth map to $M$.
	We show that the singular simplicial set $\sing(M)$ of $M$ deformation retracts onto the simplicial subset $\sing^{\mc T}(M)$ of smooth singular simplices that are transverse to every element of $\mc T$.
\end{abstract}
	\maketitle
	\tableofcontents
	\section{Introduction}

For a smooth manifold, smooth singular simplices suffice to compute homology \cite[Theorem~18.7]{Lee13}.
We show that further restricting to simplices transverse to a countable collection of manifolds with corners does so as well (\cref{C: Main Corollary}).
In fact, we prove the following more general statement at the level of simplicial sets.

\begin{theorem}\label{T: Main Theorem}
	Let $M$ be a smooth manifold without boundary and $\mc T$ a countable set of manifolds with corners with smooth maps to $M$.
	The singular simplicial set $\sing(M)$ deformation retracts onto the simplicial subset $\sing^{\mc T}(M)$ of smooth singular simplices that are transverse to every element of $\mc T$.
\end{theorem}

Here transversality is understood stratumwise: if $X \to M$ and $Y \to M$ are smooth maps from manifolds with corners, they are transverse if the restrictions $S^k(X) \to M$ and $S^\ell(Y) \to M$ are transverse in the classical sense for all $k, \ell \geq 0$; see \cref{S: background}.

\cref{T: Main Theorem} does not extend to $\mc T$ of arbitrary cardinality.
For instance, if $\dim(M) > 0$ and $\mc T$ consists of all point inclusions of $M$, then no simplex can be transverse to every element of $\mc T$.

For $\mc T = \emptyset$, the theorem asserts that $\sing(M)$ deformation retracts onto the simplicial set of smooth singular simplices.

Passing to normalized or unnormalized chains gives the following.

\begin{corollary}\label{C: Main Corollary}
	The singular chain complex $C_*(M)$ of $M$ deformation retracts onto the subcomplex $C_*^{\mc T}(M)$ generated by smooth singular simplices that are transverse to every element of $\mc T$.
\end{corollary}

An area of application of this corollary is the construction of cohomology classes via intersection.
In its most basic form we have the following.
Let \(W \subset M\) be a closed co-oriented smooth submanifold without boundary of codimension \(d\), and let \(\mc T = \{W\}\).
For a smooth singular \(d\)-simplex \(\sigma \colon \Delta^d \to M\) transverse to \(W\), define
\[
\iota_W(\sigma) =
\#\bigl(\Delta^d \times_M W\bigr)
\]
as the signed count of intersection points, and extend linearly to obtain
\[
\iota_W \in \Hom(C_d^{\mc T}(M),\Z).
\]
To see that \(\iota_W\) is a cocycle, let \(\tau \in C^{\mc T}_{d+1}(M)\).
The value of \(\iota_W\) on \(\partial \tau\) is the signed count of intersections of \(W\) with the restriction of \(\tau\) to the boundary of \(\Delta^{d+1}\).
This intersection is an oriented zero-manifold which is the boundary of the one-manifold \(\sigma^{-1}(W)\), and hence its total count is zero.

\begin{corollary}\label{C: intersection class}
	The cocycle \(\iota_W \in \Hom(C_d^{\mc T}(M),\Z)\) extends along a chain homotopy equivalence \(p_* \colon C_*(M) \to C_*^{\mc T}(M)\) to a singular cocycle \(p^*(\iota_W) \in C^d(M; \Z)\).
\end{corollary}

The resulting cohomology class, coinciding when \(M\) is compact and oriented with the Poincaré dual of the fundamental class of \(W\), is traditionally obtained by pulling back the Thom class of a tubular neighborhood of \(W\) \cite[VI.11]{Bredon1993}.
Our present approach realizes this construction directly at the cochain level, requiring no orientability or compactness from the target.
This approach fits naturally into the framework of geometric cohomology \cite{FMS-foundations}.
Indeed, let \(\mc T\) be a countable collection of proper co-oriented maps to \(M\) from manifolds with corners, and let \(C^*_{\Gamma, \mc T}(M)\) be the geometric cochain subcomplex they generate \cite{FMS-foundations}.
\cref{C: Main Corollary} yields a chain homotopy inverse \(p_*\), and hence a cochain map to singular cochains
\begin{equation*}
	C^*_{\Gamma, \mc T}(M) \xr{\phi} \Hom(C_*^{\mc T}(M),\Z) \xr{p^*} \Hom(C_*(M),\Z)
\end{equation*}
where \(\phi\) is evaluation on transverse chains \cite[Theorem 7.3, Definition
6.11, and Proposition 3.77]{FMS-foundations}.
In particular, the value of \(f \colon X \to M\) on a simplex \(\sigma\) of complementary dimension is the signed count of \(f \times_M p(\sigma)\).

\medskip
While our applications are in geometry and topology, this result was also motivated by questions in arithmetic.
We thank K.~Kallal and A.~Venkatesh for bringing this to our attention, as a need for \cref{T: Main Theorem} arises in their approach to group cohomology and theta functions.

	\section{Preliminaries}\label{S: background}

\subsection{Manifolds with corners}

Let $X$ be an $n$-dimensional \defn{manifold with corners} as defined in Joyce \cite{Joy12}.
Specifically, $X$ possesses an atlas of smoothly compatible local charts consisting of pairs $(U, \phi)$, where $U$ is an open subset of $\R^n_k \defeq [0,\infty)^k \times \R^{n-k}$ for some $0 \leq k \leq n$ and $\phi \colon U \to X$ is a homeomorphism onto its image.

A map $f \colon X \to Y$ between manifolds with corners is a \defn{smooth map} if for every pair of charts $(U, \phi)$ on $X$ and $(V,\psi)$ on $Y$, the composition
\[
\psi^{-1} \circ f \circ \phi \colon U \to V
\]
extends to a smooth map between open subsets of Euclidean spaces.
When the codomain has no boundary, this definition coincides with both the notions of \textit{smooth} and \textit{weakly smooth} maps in \cite{Joy12}.

For $k \geq 0$, the \defn{stratum} $S^k(X)$ is the set of points of depth $k$, that is, those points $x \in X$ for which, in some (and hence any) chart $(U, \phi)$ with $\phi(u) = x$, exactly $k$ of the coordinates of $u \in \R^n_k$ are zero.
Thus $S^0(X)$ is the interior of $X$, and if $X$ is a manifold with boundary then $S^1(X)$ is its boundary.

The \defn{geometric $n$-simplex} $\gsimplex^n$ is the convex hull of $\vec 0$ and the standard basis of $\R^n$.
It is a manifold with corners whose stratum $S^k(\gsimplex^n)$ consists of points lying in exactly $k$ boundary hyperplanes.
More explicitly, $S^k(\gsimplex^n)$ is the disjoint union of the relative interiors of all codimension $k$ faces of $\gsimplex^n$.
For example,
\[
S^0(\gsimplex^n) = \mathring \gsimplex^n \quad\text{and}\quad
S^1(\gsimplex^n) = \coprod_i \delta_i(\mathring \gsimplex^{n-1}),
\]
where $\delta_i \colon \gsimplex^{n-1} \to \gsimplex^n$ is the $i^\th$ face inclusion.

Let $M$ be a manifold without boundary.
We say that smooth maps $f \colon X \to M$ and $g \colon Y \to M$ are \defn{transverse} if for all $k,\ell \geq 0$ the map
\[
f|_{S^k(X)} \times g|_{S^\ell(Y)} \colon S^k(X) \times S^\ell(Y) \to M \times M
\]
is transverse in the classical sense to the diagonal $\diag(M) \subset M \times M$, i.e.,
\[
D(f \times g)\bigl(T_xS^k(X) \times T_yS^\ell(Y)\bigr) + T_{(z,z)}\diag(M) = T_{(z,z)}(M \times M)
\]
whenever $f(x) = g(y) = z$ or, equivalently,
\[
Df T_xS^k(X) + Dg T_yS^\ell(Y) = T_zM,
\]
see \cite[Lemma 3.69]{FMS-foundations}.

\medskip
Throughout the text we fix a smooth manifold without boundary $M$ and a countable set $\mc T$ of manifolds with corners each equipped with a smooth map to $M$.

We say that a smooth map to $M$ is \defn{$\mc T$-transverse} if it is transverse to every element in $\mc T$.

\subsection{Simplicial sets}

The \defn{simplex category} is denoted by  \(\simplex\), its objects are the posets \([n] = \set{0 < \dots < n}\) and its morphisms are all (weakly) order-preserving maps \([m] \to [n]\).

A \defn{simplicial set} is a functor \(\simplex^\op \to \Set\), and a \defn{simplicial map} is a natural transformation \cite{GBF-SS}.
We denote this category by \(\sSet\) and consider it equipped with the cartesian product.

The \(n^\th\) \defn{representable simplicial set} \(\simplex^n\) is the functor determined by \([m] \mapsto \simplex\big([m], [n]\big)\).

A \defn{simplicial homotopy} between $f, g \colon X \to Y$ is a simplicial map
\[
H \colon \simplex^1 \times X \to Y
\]
restricting to \(f\) and \(g\) on the endpoints of \(\simplex^1\).

The \defn{singular simplicial set} \(\sing(\mc X)\) of a topological space \(\mc X\) is defined by
\[
[n] \mapsto \sing_n(\mc X) \coloneqq \set{\sigma \colon \gsimplex^n \to \mc X \mid \sigma \text{ continuous}},
\]
on objects, and on morphisms \(\alpha \colon [m] \to [n]\) by
\[
\sing_n(\mc X) \ni \sigma \mapsto \sigma \circ \alpha_*\,,
\]
where \(\alpha_* \colon \gsimplex^m \to \gsimplex^n\) is the canonical affine map induced by \(\alpha\).

For a smooth manifold \(M\), we define the \defn{$\mc T$-transverse singular simplicial set} \(\sing^{\mc T}(M) \subset \sing(M)\) to be the simplicial subset consisting of all smooth \(\mc T\)-transverse simplices.
To see that this is indeed a simplicial subset, note first that transversality is preserved under restriction to faces, so the collection is closed under face maps.
For degeneracies, recall that for any surjection \(\alpha \colon [m] \to [n]\), the induced map \(\alpha_* \colon \gsimplex^m \to \gsimplex^n\) is affine and restricts to a submersion on every face.
It follows that if \(\sigma \colon \gsimplex^n \to M\) is \(\mc T\)-transverse, then \(\sigma \circ \alpha_*\) is too.
	\section{Smoothness}\label{s:smoothness}

\begin{lemma}\label{L: smoothing}
	If $\sigma \colon \gsimplex^n \to M$ is a singular simplex whose restriction to each proper face is smooth, then $\sigma$ is homotopic relative to $\bd \gsimplex^n$ to a smooth singular simplex.
\end{lemma}

\begin{proof}
	Consider a singular simplex $\sigma \colon \gsimplex^n \to M$ whose restriction to each codimension $1$ face is smooth.
	We want to construct a homotopy relative to $\bd \gsimplex^n$ from $\sigma$ to a smooth singular simplex.

	Recall that $\gsimplex^n$ is the convex hull in $\R^n$ of the origin and the standard basis vectors.
	We extend $\sigma$ to an open neighborhood $\mc U$ of $\gsimplex^n$ in $\R^n$ by composing with a collapse map from an open regular neighborhood of $\gsimplex^n$ onto $\gsimplex^n$, and continue to denote this extension by $\sigma$.
	It suffices to verify that $\sigma$ is smooth on $\bd \gsimplex^n$ in the sense that for each $x \in \bd \gsimplex^n$ there exists a smooth map defined on a neighborhood of $x$ in $\R^n$ agreeing with $\sigma$ on $\bd \gsimplex^n$.
	Under this assumption, since $\mc U$ is a manifold, the Whitney Approximation Theorem \cite[Theorem 6.26]{Lee13} yields a homotopy rel $\bd \gsimplex^n$ from $\sigma$ to a smooth map $\sigma' \colon \mc U \to M$ (see also \cite[page 45]{Lee13}).

	Suppose $z \in \bd \gsimplex^n$.
	Then $z$ has an open neighborhood $N$ in $\R^n$ that is mapped to a neighborhood of the origin in $\R^n$ by an affine map taking $z$ to the origin and $N \cap \gsimplex^n$ to an open subset of $\R^n_k \defeq [0, \infty)^k \times \R^{n-k}$ for some $k$, $1 \leq k \leq n$; note that $k \neq 0$ as $z \in \bd \gsimplex^n$.
	By shrinking $N$ and applying elementary diffeomorphisms $(-1,1) \to \R$ in each factor, we may thus obtain a diffeomorphism $\phi \colon \R^n \to N$ that takes the origin to $z$ and $\R^n_k$ to $N \cap \gsimplex^n$.
	For $1 \leq j \leq k$, let $H_j = \{(x_1, \dots, x_n) \in \R^n \mid x_j = 0\}$ and $W_j = H_j \cap \R^n_k$.
	Then $\phi$ takes each $W_j$ to the intersection of $N$ with a codimension $1$ face of $\gsimplex^n$.
	In particular, $(\R^n_k, \phi)$ is a smooth chart for $\gsimplex^n$, and the further restriction to each $W_j$ is a chart for a codimension $1$ face of $\gsimplex^n$.

	By assumption, $\sigma$ is smooth on each codimension $1$ face of $\gsimplex^n$.
	This means that the restriction of $\sigma \circ \phi$ to each $W_j$ extends to a smooth function on a neighborhood $U_j$ of the origin in the hyperplane $H_j$, $1 \leq j \leq k$.
	Let us denote by $\tilde \sigma_j$ the resulting smooth function $H_j \supset U_j \xr{\tilde \sigma_j} M$.
	This means that for any Euclidean chart $(\R^m,\psi)$ of $M$ such that $\psi(\R^m)$ is a neighborhood of $\sigma(z)$ in $M$, each composition $\psi^{-1}\tilde \sigma_j \colon U_j \to \R^m$ is smooth.
	Without loss of generality we will suppose $U_j = H_j$.

	Summarizing, we now have a map $f \defeq \psi^{-1}\sigma \phi \colon \R^n_k \to \R^m$ that represents $\sigma$ in local charts.
	Furthermore, for each $1 \leq j \leq k$, we have a \emph{smooth} map $f_j \defeq \psi^{-1} \tilde \sigma_j \colon H_j \to \R^m$ whose restriction to $W_j$ agrees with $f$ there.
	It suffices now to demonstrate that the restriction of $f$ to $W_1 \cup \dotsb \cup W_k$ is smooth, for which it further suffices to find a smooth map $F \colon \R^n \to \R^m$ such that $F$ restricts to $f_j$ on each $W_j$.
	As a map to $\R^m$ is smooth if and only if the map to each coordinate is smooth, without loss of generality we may take $m = 1$ in the remainder of the argument.

	Now consider a non-empty set $J = \{j_1 < j_2 < \cdots < j_r\} \subseteq \{1,\dots,k\}$ and $j \in J$.
	Let $f^j_J \colon \R^n \to \R$ denote the composition
	\[
	f^j_J \colon \R^n \to H_{j_1} \cap\dotsb\cap H_{j_r} \xrightarrow{f_{j}} \R
	\]
	where the first map is the canonical projection.
	Observe that the restriction of $f^j_J$ to $\R^n_k$ is independent of $j \in J$ as the $f_j$ all agree with $f$ on $W_{j_1} \cap\dotsb\cap W_{j_r}$.
	In this context we simplify notation from $f^j_J$ to $f_J$.

	We claim that the following function $F \colon \R^n \to \R$ has the desired properties:
	\[
	F = \ -\sum_{\ \mathclap{\substack{J \subseteq \{1,\dots,k\} \\ J \neq \emptyset}}} \ (-1)^{|J|}f^{\min J}_{J}.
	\]
	This function is clearly smooth as a sum of smooth functions, so it only remains to show that the restriction to each $W_i = H_i \cap \R^n_k$ agrees with $f_i$ in that domain.

	Restricting $F$ to $W_i$ implies the restriction of each summand $f_J^{\min J}$ to $\R^n_k$, making them independent of their superscript.
	Therefore,
	\begin{align*}
	F|_{W_i}
	& = -\sum_{\ \mathclap{\substack{J \subseteq \{1,\dots,k\} \\ J \neq \emptyset}}} (-1)^{|J|} f_{J \cup \{i\}} \\
	& =
	-\sum_{\mathclap{\substack{J \subseteq \{1,\dots,k\} \\ i \notin J,\, J \neq \emptyset}}} (-1)^{|J|} f_{J \cup \{i\}}
	-\sum_{\mathclap{\substack{J \subseteq \{1,\dots,k\} \\ i \in J}}} (-1)^{|J|} f_J \\
	& =
	-\sum_{\mathclap{\substack{J \subseteq \{1,\dots,k\} \\ i \notin J,\, J \neq \emptyset}}} (-1)^{|J|} f_{J \cup \{i\}}
	+f_i
	-\sum_{\mathclap{\substack{L \subseteq \{1,\dots,k\} \\ i \notin L,\, L \neq \emptyset}}} (-1)^{|L|+1} f_{L \cup \{i\}} \\
	& = f_i. \qedhere
	\end{align*}
\end{proof}
	\section{Transversality}\label{s:transversality}

\begin{lemma}\label{L: transversality}
	If $\sigma \colon \gsimplex^n \to M$ is a smooth singular simplex whose restriction to each proper face is transverse to all $T \in \mc T$, then $\sigma$ is homotopic relative to $\bd \gsimplex^n$ to a smooth singular simplex that is transverse to all $T \in \mc T$.
\end{lemma}

\begin{proof}
	Suppose $M$ is a smooth manifold without boundary and that $\mc T$ is a countable set of manifolds with corners mapping smoothly to $M$.
	If $\sigma \colon \gsimplex^n \to M$ is a smooth singular simplex whose restriction to each codimension $1$ face is transverse to $S^k(T)$ for all $k \geq 0$ and $T \in \mc T$, we want to construct a homotopy relative to $\bd \gsimplex^n$ from $\sigma$ to a smooth singular simplex $\sigma'$ with $\sigma'$ transverse to all $S^k(T)$.
	Since $S^j(\gsimplex^n)$ satisfies these conditions for all $j>0$ by assumption, it suffices to construct a homotopy relative to $\bd \gsimplex^n$ that moves $S^0(\gsimplex^n) = \mathring \gsimplex^n$ into transverse position with respect to all $S^k(T)$.

	We may assume that $M$ is embedded in some Euclidean space $\R^N$ by the Whitney Embedding Theorem~\cite[page 53]{GP74} and hence admits a tubular neighborhood $M^\epsilon \subset \R^N$, for some smooth positive function $\epsilon \colon M \to (0,\infty)$.
	Then, the nearest-point projection $\pi \colon M^\epsilon \to M$ is a smooth submersion; see \cite[Section 2.3]{GP74} and \cite[Theorem 2.17 and Remark 2.18]{FMS-foundations}.

	By the Whitney Approximation Theorem \cite[Theorem 6.26]{Lee13}, we extend $\sigma$ to be smooth in an open neighborhood $\mathcal U$ of $\gsimplex^n$.
	We now proceed by modifying the proof of the Transversality Homotopy Theorem in \cite[page 70]{GP74}.

	Let $\rho \colon \mc U \to [-1,1]$ be a smooth function satisfying $\rho = 0$ on $\bd \gsimplex^n$ and $\rho > 0$ on $\mathring{\gsimplex}^n$.

	Let $\mathring D$ be the open unit ball in $\R^N$, and define
	\[
	f \colon \mc U \times \mathring D \to \R^N,
	\qquad
	f(x,s) = \sigma(x)+\rho(x)\epsilon(\sigma(x))s.
	\]
	If $x \in \bd \gsimplex^n$, then $\rho(x) = 0$, so $f(x,s) = \sigma(x)$.
	If $x \in \mathring\gsimplex^n$, then $\rho(x)>0$, and the differential in the $\mathring D$-direction is given by multiplication by the nonzero scalar $\rho(x)\epsilon(\sigma(x))$, so $f$ is a submersion at $(x,s)$.
	Moreover,
	\[
	|f(x,s) - \sigma(x)| = \rho(x)\epsilon(\sigma(x))|s| < \epsilon(\sigma(x)),
	\]
	and therefore $f(x,s)\in M^\epsilon$.
	Hence the composition
	\[
	F = \pi\circ f \colon \mc U \times \mathring D \to M
	\]
	is well defined and smooth, and its restriction to $\mathring\gsimplex^n \times \mathring D$ is a submersion.

	Fix $(T, g_T \colon T \to M) \in \mc T$ and let $g_T^k = g_T|_{S^k(T)}$.
	We consider $F \times g_T^k \colon \mc U \times \mathring D \times S^k(T) \to M \times M$.
	Suppose $(x,s,y) \in \mathring \gsimplex^n \times \mathring D \times S^k(T)$ is such that $(F \times g_T^k)(x,s,y) = (z,z) \in \diag(M)$.
	Since $F$ is a submersion on $\mathring \gsimplex^n \times \mathring D$, the image $\operatorname{D}(F \times g_T^k)T_{(x,s,y)}(\mathring \gsimplex^n \times \mathring D \times S^k(T))$ of the tangent space at $(x, s, y)$ includes all of $T_z M \times \{0\}$.
	As this subspace is complementary to the diagonal subspace of $T_zM \times T_zM$, we see that $F \times g_T^k$ is transverse to $\diag(M)$ over $\mathring \gsimplex^n \times \mathring D \times S^k(T)$.
	Thus, by the Transversality Theorem of \cite[page 68]{GP74}, for almost every $s \in \mathring D$, the map $F(\cdot, s) \times g_T^k \colon \mathring \gsimplex^n \times S^k(T) \to M \times M$ is transverse to $\diag(M)$.

	Since $\mc T$ is countable and each $T \in \mc T$ is finite-dimensional, the collection of spaces $S^k(T)$, as $T$ ranges over $\mc T$ and $k$ over $\mathbb{N}$, is countable.
	Since the union of a countable number of sets of measure zero has measure zero, almost every $s \in \mathring D$ suffices for $F(\cdot, s) \times g_T^k$ to be transverse to $\diag(M)$ for all $T \in \mc T$ and all $k \geq 0$.

	Choosing such an $s$, let $\sigma'(x) = F(x,s)$ for all $x \in \gsimplex^n$.
	Then, as required, $\sigma'$ is smooth, as it extends to the smooth $F(\cdot,s)$ on all of $\mc U$, it agrees with $\sigma$ on $\bd \gsimplex^n$ (and so by assumption its restriction to each open face of $\bd \gsimplex^n$ is transverse to all $S^k(T)$), and its restriction to the interior of $\gsimplex^n$ is transverse to all $S^k(T)$.
	Finally, $\sigma'$ is homotopic rel $\bd \gsimplex^n$ to $\sigma$ via the homotopy $F(\cdot,ts)$, noting that $F(\cdot,0) = \sigma$.
\end{proof}
	\section{Proof of \cref{T: Main Theorem}}\label{s:main}

Let $i \colon \sing^{\mc T}(M) \to \sing(M)$ be the canonical inclusion.
To prove \cref{T: Main Theorem} we must construct a simplicial map
\[
p \colon \sing(M) \to \sing^{\mc T}(M)
\]
satisfying
\[
p \circ i = \id_{\sing^{\mc T}(M)},
\]
and a simplicial homotopy
\[
H \colon \simplex^1 \times \sing(M) \to \sing(M)
\]
between $\id_{\sing(M)}$ and the composite $i \circ p$.

Assume for the moment that, for every singular simplex \(\sigma \colon \gsimplex^n \to M\), we are given a homotopy
\[
h_\sigma \colon \gsimplex^1 \times \gsimplex^n \to M
\]
from \(\sigma\) to a simplex \(p(\sigma) \in \sing^{\mc T}(M)\) such that, for every morphism \(\beta \colon [m] \to [n]\) in \(\simplex\),
\[
h_{\sigma \circ |\beta|} = h_\sigma \circ (\id_{\gsimplex^1} \times |\beta|).
\]
For \(\beta = \id_{[n]}\), this condition is vacuous.
For face inclusions, it says that the homotopies on the faces of \(\sigma\) are the restrictions of the homotopy of \(\sigma\).
For degeneracy maps, it says that the homotopy of a degenerate simplex is obtained by precomposing the homotopy with the corresponding simplex collapse.
We assume moreover that if $\sigma \in \sing^{\mc T}(M)$ then $h_\sigma$ is constant.

The simplicial map
\[
p \colon \sing(M) \to \sing^{\mc T}(M)
\]
is defined by the assignment
\[
(\sigma \colon \gsimplex^n \to M) \mapsto \big(p(\sigma) \colon \gsimplex^n = \gsimplex^0 \times \gsimplex^n \xr{\delta_1 \times \id} \gsimplex^1 \times \gsimplex^n \xr{h_\sigma} M\big).
\]

Indeed, for every $\beta \colon [m] \to [n]$,
\[
p(\sigma \circ |\beta|) = h_{\sigma \circ |\beta|} \mid_{\{1\} \times \gsimplex^m}
= h_\sigma \circ (\id_{\gsimplex^1} \times |\beta|)\mid_{\{1\} \times \gsimplex^m}
= p(\sigma) \circ |\beta|.
\]
Moreover, for every $\sigma \in \sing^{\mc T}(M)$ we have
\[
p(i(\sigma)) = p(\sigma) = \sigma.
\]

We now define the simplicial homotopy
\[
H \colon \simplex^1 \times \sing(M) \to \sing(M).
\]
An $n$-simplex of $\simplex^1 \times \sing(M)$ is a pair $\alpha \times \sigma$, where $\alpha \colon \simplex^n \to \simplex^1$ is a simplicial map and $\sigma \colon \gsimplex^n \to M$ is a singular simplex.
We set
\[
H(\alpha \times \sigma) \coloneqq
\bigl(
\gsimplex^n \xrightarrow{\mathrm{diag}\,}
\gsimplex^n \times \gsimplex^n \xr{(|\alpha| \times \id)} \gsimplex^1 \times \gsimplex^n \xr{h_\sigma} M
\bigr).
\]
Then, since each $h_\sigma$ begins at $\sigma$ and ends at $p(\sigma)$,
\[
H \circ (\delta_0 \times \id_{\sing(M)}) = \id_{\sing(M)}
\qquad\text{and}\qquad
H \circ (\delta_1 \times \id_{\sing(M)}) = i \circ p,
\]
where we are using the identification \(\gsimplex^0 \times \gsimplex^n = \gsimplex^n\) implicitly.

To verify that $H$ is simplicial, let $\beta \colon [m] \to [n]$ be a morphism in $\simplex$.
Then
\begin{align*}
	H(\alpha \times \sigma)\circ |\beta|
	&= h_\sigma \circ (|\alpha| \times \id_{\gsimplex^n}) \circ |\beta| \\
	&= h_\sigma \circ (|\alpha|\circ |\beta| \times |\beta|) \\
	&= h_{\sigma \circ |\beta|} \circ (|\alpha|\circ |\beta| \times \id_{\gsimplex^m}) \\
	&= H(\alpha \circ \beta \times \sigma \circ |\beta|),
\end{align*}
as required.
It therefore remains to construct such a natural family of homotopies $h_\sigma$.

\medskip
To construct \(h_\sigma\), it suffices to define it for non-degenerate simplices and verify compatibility with face maps.
Indeed, every simplex \(\sigma \colon \gsimplex^n \to M\) admits a unique factorization
\[
\sigma = \bar{\sigma} \circ s,
\]
where \(\bar{\sigma} \colon \gsimplex^m \to M\) is non-degenerate and \(s \colon \gsimplex^n \to \gsimplex^m\), if present, is induced by a composite of degeneracy maps.
Given \(h_{\bar{\sigma}} \colon \gsimplex^1 \times \gsimplex^m \to M\) we extend to \(\sigma\) by
\[
h_\sigma \coloneqq h_{\bar{\sigma}} \circ (\id \times s).
\]

For \(\sigma\) non-degenerate we distinguish two cases.
If $\sigma$ is already smooth and $\mc T$-transverse, we take $h_\sigma$ to be the constant homotopy.
Otherwise, we proceed by induction on $n$, the dimension of $\sigma$.

In the base case $n = 0$, every $\sigma$ is automatically smooth.
By \cref{L: transversality}, there is a homotopy on $[0,1/2]$ deforming $\sigma$ to a $\mc T$-transverse singular simplex.
We extend this homotopy to $[1/2,1]$ by constancy.

Now suppose $n>0$.
By induction, we have homotopies $h_{\sigma \circ \delta_i}$ which assemble into a homotopy
\[
h_{\partial \sigma} \colon [0,1] \times \bd \gsimplex^n \to M,
\]
assumed constant for $t \geq 1-\frac{1}{n+1}$.

On $\big[0,1-\tfrac{1}{n+1}\big] \times \gsimplex^n$, choose a retraction
\[
r \colon \big[0,1-\tfrac{1}{n+1}\big] \times \gsimplex^n \to \big(\{0\} \times \gsimplex^n\big) \cup \Big(\big[0,1-\tfrac{1}{n+1}\big] \times \bd \gsimplex^n\Big),
\]
and define $h_\sigma$ as the composition of $r$ with the map that agrees with $\sigma$ on $\{0\} \times \gsimplex^n$ and with $h_{\partial \sigma}$ on $\big[0,1-\tfrac{1}{n+1}\big] \times \bd \gsimplex^n$.
The restriction of $h_\sigma$ to $\{1-\tfrac{1}{n+1}\} \times \gsimplex^n$ determines a singular simplex $\sigma'$ whose proper faces are smooth and $\mc T$-transverse.

On $[1-\tfrac{1}{n+1},\,1-\tfrac{1}{n+2}] \times \gsimplex^n$, we proceed in two stages.
On the first half of this interval, we apply \cref{L: smoothing} to deform $\sigma'$ to a smooth simplex $\sigma''$.
On the second half, we apply \cref{L: transversality} to deform $\sigma''$ to a $\mc T$-transverse simplex $\sigma'''$.

Finally, on $[1-\tfrac{1}{n+2},\,1] \times \gsimplex^n$ we define $h_\sigma$ to be the constant homotopy at $\sigma'''$.

Since the homotopies constructed by \cref{L: smoothing} and \cref{L: transversality} leave the boundary invariant, we have $h_{\sigma \circ \delta_j} = h_\sigma \circ ([0,1] \times \delta_j)$ for every $j$, as claimed. \hfill\qedsymbol\par
	\sloppy
	\bibliographystyle{alpha}
	\bibliography{auxy/bibliography}
\end{document}